\documentclass[12pt]{amsart}
\usepackage{latexsym,rawfonts}
\usepackage{amsfonts,amssymb}
\usepackage{amsmath,amsthm}

\usepackage[plainpages=false]{hyperref}
\usepackage{graphicx}
\usepackage{color}

\newtheorem{theorem}{Theorem}
\newtheorem{corollary}{Corollary}

\newtheorem{lemma}{Lemma}

\theoremstyle{definition}

\newtheorem{remark}{Remark}

\newcommand{\R}{\mathbb{R}}

\newcommand{\dd}{\mathop{}\!\mathrm{d}}

\newcommand{\set}[1]{\left\{#1\right\}}

\newcommand{\pd}{\partial}

\newcommand{\uS}{\mathbb{S}^{n-1}}

\newcommand{\beq}{\begin{equation}}
\newcommand{\eeq}{\end{equation}}
\newcommand{\beqs}{\begin{eqnarray*}}
\newcommand{\eeqs}{\end{eqnarray*}}
\newcommand{\beqn}{\begin{eqnarray}}
\newcommand{\eeqn}{\end{eqnarray}}

\begin{document}

\title{A flow method for a generalization of  $L_{p}$ Christofell-Minkowski problem}

\author{BoYa Li}
\address{BoYa Li: School of Mathematics and Statistics, Beijing Technology and Business University, Beijing 100048, P.R. China}
\email{liboya\_btbu@163.com}

\author{HongJie Ju}
\address{HongJie Ju: School of Science,Beijing University of Posts and Telecommunications, Beijing 100876, P.R. China}
\email{hjju@bupt.edu.cn}

\author{ YanNan Liu}
\address{YanNan Liu: School of Mathematics and Statistics, Beijing Technology and Business University, Beijing 100048, P.R. China}
\email{liuyn@th.btbu.edu.cn}

\thanks{This work was supported in part by
the Natural Science Foundation of Beijing Municipality (No. 1212002) and
the National
Natural Science Foundation of China (Grant No.12071017, 11871432, 11871102).}

\date{}

\begin{abstract}
In this paper, a generalization of the  $L_{p}$-Christoffel-Minkowski problem
 is studied. We consider an anisotropic curvature flow and derive the long-time existence of the flow. Then under some initial data, we  obtain the existence  of smooth solutions to  this problem for $c=1$.
\end{abstract}




\maketitle
\vskip4ex

\section{Introduction}
Recently, as important developments of the classical Brunn-Minkowski theory in convex geometry,  the dual Brunn-Minkowski theory is
developing rapidly.  Dual curvature measures and  their associated
variational formulas was firstly introduced by  Huang, Lutwak,
Yang and Zhang in their recent groundbreaking work \cite{HuangActa2016}.
Then the dual curvature measures were extended into the $L_p$ case in
\cite{LutwakYangZhang2018} and then studied in \cite{BF2019,Chen&Huang&Zhao2018,HuangZhao2018}.

In this paper, we consider the following fully nonlinear equation, which is a generalization of $L_{p}$-Christoffel-Minkowski problem
\begin{equation} \label{eq}
\frac{ h^{1-p}}{(h^{2} + |\nabla h|^{2})^{\frac{n-q}{2}}} \sigma_{k}(x) f(x)=  c\text{ on } \uS
\end{equation}
for some positive constant $c$.
Here  $f$ is a given positive and smooth function on the unit sphere
 $\uS$ and $h$ is the support function defined on $\uS$.
$\sigma_{k}(x,t)$ is the
$k$-th  elementary symmetric function for principal curvature radii and $\nabla$ is  the  Levi-Civita connection on  $\uS$.

 Equation \eqref{eq} is just the smooth case of $L_{p}$ dual Minkowski problem when $k = n-1$.

When $q=n, k=n-1$, Eq. \eqref{eq} reduces to the $L_p$-Minkowski
problem, which has been extensively studied, see e.g. Schneider' book \cite{Schneider}.
For $q=n, 1\leq k< n-1$, Eq. \eqref{eq}
is known as the $L_{p}$-Christoffel-Minkowski problem and is the  classical Christoffel-Minkowski problem for $p=1$.   Under a sufficient condition on the prescribed function, existence of solution for the classical Christoffel-Minkowski problem was given  in \cite{GuanMa2003}.

The $L_{p}$-Christoffel-Minkowski problem is related to the problem of prescribing $k$-th $p$-area measures. Hu, Ma \& Shen in  \cite{HMS2004} proved the existence of convex solutions to the $L_{p}$-Christoffel-Minkowski problem  for $ p \geq k+1$ under appropriate conditions. Using the methods of geometric flows, Ivaki in \cite{Ivaki2019} and then  Sheng \& Yi in \cite{ShengYi2020} also gave  the existence of smooth convex solutions to the $L_{p}$-Christoffel-Minkowski problem  for $ p \geq k+1$.
In case $1 <  p < k+1$, Guan \& Xia in  \cite{GuanXia2018} established the existence of convex body with prescribed $k$-th even $p$-area  measures. In \cite{BIS2021}, the authors considered a generalized $L_{p}$-Christoffel-Minkowski problem and gave the the existence of smooth solutions by curvature flow method. To the best of our knowledge, there is no other existence result about Eq. \eqref{eq}

The existence of smooth solutions  for Eq.  \eqref{eq} is  concerned in this paper. In general, there is no variational structure, We use a flow method involving  $\sigma_{k}$, support function and radial function to give the
existence of smooth solutions  for Eq.  \eqref{eq} wich $c=1$.

Let $M_{0}$ be a smooth, closed, strictly convex hypersurface in the Euclidean
space $\R^{n}$, which encloses the origin and is given by a smooth embedding
$X_{0}: \mathbb{S}^{n-1} \rightarrow \R^{n}$.
Consider a family of closed hypersurfaces $\set{M_t}$ with $M_t=X(\uS,t)$, where $X:
\mathbb{S}^{n-1}\times[0,T) \rightarrow \R^{n}$ is a smooth map satisfying the
following initial value problem:
\begin{equation}\label{feq}
  \begin{split}
    \frac{\pd X}{\pd t} (x,t)
    &= f(\nu)\sigma_{k}(x,t) \langle X, \nu \rangle^{2-p}|X|^{q-n} \nu - \eta(t)X,\\
    X(x,0) &= X_{0}(x).
  \end{split}
\end{equation}
Here  $\nu$ is the unit outer
normal vector of  $M_t$ at the point $X(x,t)$.  $\langle \cdot,\cdot \rangle$ is the
standard inner product in $\R^n$, $\eta$ is a scalar function to be specified later, and $T$ is the maximal time
for which the solution exists. We use $\{e_{ij}\}, 1 \leq i, j \leq n-1$ and $\nabla$ for the standard metric and Levi-Civita connection of  $\uS$ respectively. Principal radii
of curvature are the eigenvalues of the matrix
\beqs b_{ij} := \nabla_{i}\nabla_{j}h  + e_{ij}h \eeqs
with respect to $\{e_{ij}\}$.
$\sigma_{k}(x,t)$ is the
$k$-th  elementary symmetric function for principal curvature radii of $M_{t}$ at $X(x,t)$
and $k$ is an integer with $1\leq k < n-1$. In this paper,
$\sigma_{k}$ is normalized so that $\sigma_{k} (1, \ldots, 1) = 1$.

Geometric flows with speed of symmetric polynomial of the principal curvature radii of the hypersurface have been  extensively studied, see e.g. \cite{Urbas1991}.

On the other hand,  anisotropic curvature flows provide
alternative methods to prove the existences of elliptic PDEs arising from  convex geometry,  see e.g. \cite{BIS2019,Chen&Huang&Zhao2018,Ivaki2019,LiuLuTrans,ShengYi2020}.

The scalar function $\eta(t)$ in \eqref{feq} is usually used to  keep $M_t$ normalized in a certain sense, see for examples \cite{Chen&Huang&Zhao2018,Ivaki2019,ShengYi2020}.
In this paper, $\eta$ is  given by
\begin{equation}\label{eta}
  \eta(t) = \frac{\int_{\uS} \rho^{q-n}h\sigma_{k}\dd x}{ \int_{\uS}\frac{1}{f(x)}h^{p}\dd x},
\end{equation}
where $h(\cdot,t)$ and $\rho(u,t)$ are  the support
function and radial function of the convex hypersurface $M_t$.

To obtain the long-time existence of flow \eqref{feq}, we need some constraints on $f$.

${\rm \bf (A)}$: Let $s$ be the arc-length parameter  and $f$ be a smooth function on $\uS$. On every great circle, $f$  satisfies
\begin{equation*}f_{ss} - \frac{1}{k+1} \left(k  +  \frac{q-n}{q-n-k-1}+ \frac{p-2}{p+k-1}\right)f_{s}^{2}f^{-1} + (p+k-1-q+n)f > 0.
\end{equation*}

The main results of this paper are stated as  follows.

\begin{theorem} \label{thm1}
  Assume $M_{0}$ is a smooth, closed and strictly convex hypersurface in $\R^{n}$. Suppose $k$ is an integer with $1 \leq k < n-1$ and $k+1 < q-n < p-k-1 $.
Suppose $f$ is a smooth positive function on $\uS$ satisfying  ${\rm \bf (A)}$.
 Then flow \eqref{feq}
  has a unique smooth solution $M_{t}$ for all time $t > 0$.
  Moreover, when $t\to\infty$, a subsequence of $M_{t}$ converges in $C^{\infty}$ to a smooth,
  closed,  strictly convex hypersurface.
\end{theorem}

For the proof of Theorem \ref{thm1}, we will see  it is enough to
 obtain the uniform positive
upper and lower bounds for support functions of $\set{M_t}$ under condition $0\leq q-n < p-k-1 $. And the stronger condition $k+1 < q-n < p-k-1 $ is only used to derive the uniform bound of principal curvature, see Lemma \ref{C2-estimate-2} for details.

\begin{corollary} \label{coreq}
  Under the assumptions of Theorem \ref{thm1},
  there exists a smooth  solution to equation \eqref{eq} with  $c=1$.
\end{corollary}

This paper is organized as follows.
In section 2, we give some basic knowledge about the flow \eqref{feq} and evolution equations of some geometric quantities.
In section 3, the long-time existence of flow \eqref{feq} will be obtained.
First, we obtain the uniform positive
upper and lower bounds for support functions of $\set{M_t}$.
Based on the bounds of support functions, we obtain the uniform bounds of  principal curvatures
by constructing proper auxiliary functions.
The long-time existence of flow \eqref{feq} then follows by standard arguments.
In section 4, under some special initial condition, we prove that
a subsequence of $\set{M_t}$ converges to a smooth solution to equation
\eqref{eq} with $c=1$, completing the proofs of Corollar \ref{coreq}.

\section{Preliminaries}
Let $\R^n$ be the $n$-dimensional Euclidean space, and $\uS$ be the unit sphere
in $\R^n$.
Assume $M$ is a smooth closed strictly convex hypersurface in $\R^{n}$.
Without loss of generality, we may assume that $M$ encloses the origin.
The support function $h$ of $M$ is defined as
\begin{equation*}
h(x) := \max_{y\in M} \langle y,x \rangle, \quad \forall x\in\uS,
\end{equation*}
where $\langle \cdot,\cdot \rangle$ is the standard inner product in $\R^n$.

Denote the Gauss map of $M$  by $\nu_M$.
Then $M$ can be parametrized by the inverse Gauss map $X :
\mathbb{S}^{n-1}\rightarrow M$ with $X(x) =\nu_M^{-1}(x)$.
The support function $h$ of $M$ can be computed by
\begin{equation} \label{h}
  h(x) = \langle x, X(x)\rangle, \indent x \in \mathbb{S}^{n-1}.
\end{equation}
Note that $x$ is just the unit outer normal of $M$ at $X(x)$.
Differentiating \eqref{h}, we have
\begin{equation*}
  \nabla_{i} h = \langle \nabla_{i}x, X(x)\rangle + \langle x, \nabla_{i}X(x)\rangle.
\end{equation*}
  Since $\nabla_{i}X(x)$ is tangent to $M$ at $X(x)$, we have
\begin{equation*}
  \nabla_{i} h = \langle \nabla_{i}x, X(x)\rangle.
\end{equation*}
It follows that
\begin{equation}\label{Xh}
  X(x) = \nabla h + hx.
\end{equation}

By differentiating \eqref{h} twice, the  second fundamental form $A_{ij}$   of $M$
can be computed in terms of the support function, see for example \cite{Urbas1991},
\begin{equation}
\label{A} A_{ij} =  \nabla_{ij}h + he_{ij},
\end{equation}
where $\nabla_{ij} = \nabla_{i}\nabla_{j}$ denotes the second order covariant derivative with respect to $e_{ij}$.
The  induced metric matix $g_{ij}$ of $M$ can be derived by Weingarten's formula,
\begin{equation}
  \label{g}
  e_{ij} = \langle \nabla_{i}x, \nabla_{j}x\rangle  = A_{im}A_{lj}g^{ml}.
\end{equation}
The principal radii of curvature are the eigenvalues of the matrix $b_{ij} =
A^{ik}g_{jk}$.
When considering a smooth local orthonormal frame on $\uS$, by virtue of
\eqref{A} and \eqref{g}, we have
\begin{equation}
  \label{radii}
  b_{ij} = A_{ij} = \nabla_{ij}h + h\delta_{ij}.
\end{equation}
We will use
$b^{ij}$ to denote the inverse matrix of $b_{ij}$.

From the evolution equation of $X(x,t)$ in flow \eqref{feq}, we derive the
evolution equation of the corresponding support function $h(x,t)$:
\begin{equation}\label{seq}
\frac{\pd h(x,t)}{\pd t} = f(x)\sigma_{k}(x,t)h^{2-p}\rho^{q-n} - \eta(t)h(x,t).
\end{equation}

The radial function $\rho$ of $M$ is given by
\begin{equation*}
\rho(u) :=\max\set{\lambda>0 : \lambda u\in M}, \quad\forall\ u\in\uS.
\end{equation*}
Note that $\rho(u)u\in M$.

From $\eqref{Xh}$, $u$ and $x$ are related by
\begin{equation}
  \label{rs}
  \rho(u)u = \nabla h(x) + h(x)x
\end{equation}
and
\beqs \rho^{2} = |\nabla h|^{2} + h^{2}.\eeqs
Let $x = x(u,t)$, by $\eqref{rs}$, we have
\begin{equation*}
\log \rho(u,t) = \log h(x,t) - \log \langle x,u \rangle.
\end{equation*}
Differentiating the above identity, we have
\begin{equation*}
  \begin{split}
    \frac{1}{\rho(u,t)}\frac{\pd \rho(u,t)}{\pd t}
    &= \frac{1}{h(x,t)}\Bigl(\nabla h\cdot \frac{\pd x(u,t)}{\pd t} + \frac{\pd h(x,t)}{\pd t}\Bigr)
    - \frac{u}{\langle x,u \rangle} \cdot \frac{\pd x(u,t)}{\pd t}\\
    &= \frac{1}{h(x,t)}\frac{\pd h(x,t)}{\pd t}
    + \frac{1}{h(x,t)}[\nabla h - \rho(u,t)u]\cdot \frac{\pd x(u,t)}{\pd t}\\
    &= \frac{1}{h(x,t)}\frac{\pd h(x,t)}{\pd t}.
  \end{split}
\end{equation*}
The evolution equation of  radial function then follows from \eqref{seq},
\begin{equation}\label{req}
\frac{\pd \rho}{\pd t} (u,t)= f(x)\sigma_{k}(u,t)h^{1-p}\rho^{q-n+1} -  \eta(t)\rho(u,t),
\end{equation}
where $\sigma_{k}(u,t)$ denotes the fundamental symmetric function of principal radii at $\rho(u,t)u \in M_{t}$ and
$f$ takes value at the unit normal vector $x(u,t)$.

 In the rest of the paper, we take a local orthonormal frame $\{e_{1}, \cdots, e_{n-1}\}$ on $\mathbb{S}^{n-1} $ such that the standard metric on $\mathbb{S}^{n-1} $ is $\{\delta_{ij}\}$.  Double indices
 always mean to sum from $1$ to $n-1$. We  denote partial derivatives $ \dfrac{\partial \sigma_{k}}{\partial b_{ij}}$ and $ \dfrac{\partial^{2} \sigma_{k}}{\partial b_{ab}\partial b_{mn}}$ by $\sigma_{k}^{ij}$ and $\sigma_{k}^{ab,mn}$  respectively.  For convenience, we also write
\begin{equation*}
\begin{split}
N &= f(x)h^{2-p}\rho^{q-n}, \\
F & = N\sigma_{k}.
\end{split}
\end{equation*}

By the flow equation \eqref{feq}, we can derive evolution equations of some geometric quantities.

 \begin{lemma}\label{evolutions}
The following evolution equations hold along the flow \eqref{feq}.
\begin{equation*}
\begin{split}
& \partial_{t}b_{ij} -  N\sigma_{k}^{ab}\nabla_{ab}b_{ij} \\
 & = (k+1) N\sigma_{k}\delta_{ij} -  N\sigma_{k}^{ab}\delta_{ab}b_{ij} + N(\sigma_{k}^{ia}b_{ja} - \sigma_{k}^{ja}b_{ia})\\
& +  N\sigma_{k}^{ab,mn}\nabla_{j}b_{ab}\nabla_{i}b_{mn} + \left(\sigma_{k}\nabla_{ij}N   + \nabla_{j}\sigma_{k}\nabla_{i}N
+ \nabla_{i}\sigma_{k}\nabla_{j}N\right) - b_{ij}\eta(t)\\
& \partial_{t}b^{ij} -  N\sigma_{k}^{ab}\nabla_{ab}b^{ij} \\
 & =- (k+1) N\sigma_{k}b^{ia}b^{ja} +  N\sigma_{k}^{ab}\delta_{ab}b^{ij} - Nb^{ia}b^{jb}(\sigma_{k}^{ra}b_{rb} - \sigma_{k}^{rb}b_{ra})\\
& -  Nb^{il}b^{js}(\sigma_{k}^{ab,mn} + 2\sigma_{k}^{am}b^{nb})\nabla_{l}b_{ab}\nabla_{s}b_{mn}\\
& -b^{ia}b^{jb}( \sigma_{k}\nabla_{ab}N  + \nabla_{b}\sigma_{k}\nabla_{a}N
+ \nabla_{a}\sigma_{k}\nabla_{b}N) + \eta(t)b^{ij}.
\end{split}
 \end{equation*}
 \end{lemma}
For the specific computations, one can refer to  Lemma 2.3 in \cite{Ivaki2019}.
\section{The long-time existence of  the flow }

In this section, we will give a priori estimates about support functions and curvatures to
obtain the long-time existence of flow \eqref{feq} under assumptions of Theorem \ref{thm1}.

In the rest of this paper, we  assume that $M_{0}$ is a smooth,
closed,  strictly convex hypersurface in $\R^{n}$ and
$h: \uS\times[0,T)\to \R$ is a smooth solution to the Eq. \eqref{seq}
with the initial $h(\cdot,0)$ the support function of $M_0$.
Here $T$ is the maximal time for which the solution exists.
Let $M_t$ be the convex hypersurface determined by $h(\cdot,t)$, and
$\rho(\cdot,t)$ be the corresponding radial function.

We first give the uniform positive upper and lower bounds of $h(\cdot,t)$ and $\rho(u,t)$ for $t\in[0,T)$.

\begin{lemma}\label{C0-estimate} Let $h$ be a smooth  solution of $\eqref{seq}$ on $\mathbb{S}^{n-1} \times [0, T)$ and $f$ be a positive, smooth function on $\uS$. $p > 1$ and  $ 0 \leq q-n < p-k-1$. Then
\beqn  \label{h1}\frac{1}{C} \leq  h(x,t) \leq C, \\
\label{rho1}\frac{1}{C} \leq  \rho(u,t) \leq C ,\eeqn
where $C$ is a positive constant independent of $t$.
\end{lemma}
\begin{proof}
Let $J(t) = \int_{\uS} h^{p}\frac{1}{f(x)}\dd x$. We claim that $J(t)$ is unchanged along the flow $\eqref{feq}$. It is because
\begin{equation*}
\begin{split}
 J'(t) &= \int_{\uS}{ph^{p-1}}\frac{1}{f(x)}\partial_{t} h \dd x \\
&=   p\int_{\uS}h^{p-1}\frac{1}{f(x)}(\sigma_{k}h^{2-p}\rho^{q-n}f(x )- \eta(t)h) \dd x \\
& = 0.
\end{split}
\end{equation*}

For each $t \in [0,T)$, suppose that the maximum of radial function $\rho(\cdot,t)$ is attained at some $u_{t} \in \uS$.
 Let
$$R_{t} = \max_{u \in \uS}\rho (u,t) = \rho (u_{t},t)$$ for some $u_{t}\in \uS$.
By the definition of support function, we have
\begin{equation*}
h(x,t)\geq R_t\langle x,u_t \rangle, \quad \forall x\in\uS.
\end{equation*}
Denote the hemisphere containing $u_t$ by $S_{u_{t}}^{+} =\set{x\in\uS : \langle x,u_t \rangle > 0}$. Since $p>1$, we have
\begin{equation*}
\begin{split}
J(0) & = J(t) \geq \int_{S_{u_{t}}^{+}} h^{p}\frac{1}{f(x)}\dd x \\
 &  \geq \int_{S_{u_{t}}^{+}}R_{t}^{p} \langle x,u_t \rangle ^{p}\frac{1}{f(x)}\dd x \geq \frac{1}{f_{\max}} \int_{S_{u_{t}}^{+}}R_{t}^{p} \langle x,u_t \rangle ^{p}\dd x\\
   &=\frac{1}{f_{\max}}  \int_{S^{+}} R_{t}^{p} x_1 ^{p} \dd x,
 \end{split}
\end{equation*}
 where $S^{+}  =\set{x\in \uS : x_1\ >  0}$.

Denote $S_1 =\set{x\in \uS : x_1\geq 1/2}$, then
\begin{equation*}
\begin{split}
  J(0)
  \geq \frac{1}{f_{\max}}  \int_{S_1} R_{t}^{p} (\frac{1}{2}) ^{p} \dd x
  = \frac{1}{f_{\max}}  R_{t}^{p} (\frac{1}{2}) ^{p} |S_1|,
\end{split}
\end{equation*}
which implies that $R_{t}$ is uniformly bounded from above.

Now, we estimate the lower bound of $h$. First we explain that $\eta(t) $ is  bounded from above.

Since mixed volumes are monotonic increasing, see \cite[page 282]{Schneider}, we
have for each $t \in [0,T)$,
 \begin{equation}
 \label{v-bound}h_{\min}^{k+1}(t) \leq \dfrac{ \int_{\uS}h\sigma_{k} \dd x }{\omega_{n-1}}\leq   h_{\max}^{k+1}(t),
\end{equation}
here $h_{\min}(t) = \min_{x \in \uS} h(x,t)$ and $h_{\max}(t) = \max_{x \in \uS} h(x,t)$.

Now we recall the definition of  $\eta(t)  $ and notice that $J(t)$ is unchanged along the flow $\eqref{feq}$. From the upper bound of $h_{\max}(t)$ and $\eqref{v-bound}$, we have
 \begin{equation*}
\eta(t)  \leq c_{3}
 \end{equation*}
for some  positive constant  $c_{3}$ independent of $t$.

Suppose the minimum of $h$ is attained at a point $(y_{t}, t)$.  It follows that
\begin{equation*}
\sigma_{k}(y_{t}, t) \geq h_{\min}^{k}(t).
 \end{equation*}
Then in the sense of the lim inf of  difference quotient, we have
\begin{equation*}
 \frac{\partial h_{\min}(t)}{\partial t}  \geq
 f_{\min}h_{\min}^{k+2-p+q-n}- c_{3}h_{\min}.
 \end{equation*}

 The right hand of the above inequality is positive for $h_{\min}(t)$ small enough if $k+1-p+q-n < 0$.
The lower bound of $h_{\min}(t)$ follows from the maximum principle.

\end{proof}

\begin{remark}When  $\eta =1$, $J(t)$ is not unchange along the flow any more. The upper bound of $h$ can be estimated as follows.

Suppose the maximum of $h(x,t)$ is attained at a point $(x_{t}, t)$. At $(x_{t}, t)$,
\begin{equation*}
\sigma_{k}(x_{t}, t) \leq h_{\max}^{k}(t).
 \end{equation*}
 Hence, we have
 \begin{equation*}
 \frac{\partial h_{\max}(t)}{\partial t}  \leq
 f_{\max}h_{\max}^{k+2-p+q-n}- h_{\max}.
 \end{equation*}
 The right hand of the above inequality becomes negative for {$h_{\max}(t)$} large enough if $k+1-p+q-n < 0$.
The upper bound of $h_{\max}(t)$ follows. The lower bound $h(x,t)$ can be obtained  by using the same method for  the case
 $\eta(t) \neq 1 $.
 \end{remark}

By the equality $\rho^{2} = h^{2} + |\nabla h|^{2}$, we can obtain the gradient estimate of support function from Lemma \ref{C0-estimate}.

\begin{corollary}\label{cor3.2}
  Under the assumptions of Lemma \ref{C0-estimate}, we have
\begin{equation*}
  |\nabla h(x,t)| \leq C, \quad \forall (x,t) \in \mathbb{S}^{n-1} \times [0, T),
\end{equation*}
where $C$ is a positive constant depending only on constants in Lemma \ref{C0-estimate}.
\end{corollary}

 The uniform bounds of $\eta(t)$ can be derived from Lemma \ref{C0-estimate}.
\begin{lemma}\label{eta-estimate}   Under the assumptions of Lemma \ref{C0-estimate}, $\eta(t)$ is uniformly bounded above and below from zero.
\end{lemma}
\begin{proof}
From the proof of Lemma \ref{C0-estimate}, we know $\eta(t)$ has  positive upper bound and
 \begin{equation*}
 h_{\min}^{k+1}(t) \leq \dfrac{ \int_{\uS}h\sigma_{k} \dd x }{\omega_{n-1}}\leq   h_{\max}^{k+1}(t).
\end{equation*}
 Lemma \ref{C0-estimate} also gives
$$\int_{\uS}\rho ^{q-n}h\sigma_{k} \dd x \geq c\int_{\uS}h\sigma_{k} \dd x,$$
here $c$ is a constant depending on the bounds of $h(x,t)$.
  This together  with the uniform of $h(x,t)$
 implies that $\eta(t)$ is bounded from below.

\end{proof}

To obtain the long-time existence of the flow $\eqref{feq}$, we need to establish the uniform bounds on principal curvatures.
By Lemma \ref{C0-estimate}, for any $t\in[0,T)$, $h(\cdot,t)$
 always ranges within a bounded interval $I'=[1/C,C]$, where $C$ is the
 constant in Lemma \ref{C0-estimate}.
 First, we  give the estimates of $\sigma_{k}$.

\begin{lemma}\label{sigma-lower-estimate}   Under the assumptions of Lemma \ref{C0-estimate},
$$ \frac{1}{C} \leq \sigma_{k}(x,t) \leq C, \quad \forall (x,t) \in \mathbb{S}^{n-1} \times [0, T),$$
where $C$ is a positive constant independent of $t$.
\end{lemma}

\begin{proof}
Recall that $F = h^{2-p}\rho^{q-n}\sigma_{k}f(x)$. According to   Lemma \ref{C0-estimate}, uniform bounds of $\sigma_{k}$ will follow from those of
$F$.

First, we compute the evolution equation of $F$. Since
\begin{equation*}
\begin{split}
\partial_{t}\sigma_{k} & =\sigma_{k}^{ij}\partial_{t}(\nabla_{ij}h + \delta_{ij}h) \\
&=  \sigma_{k}^{ij}\nabla_{ij}(\partial_{t}h) +  \sigma_{k}^{ij}\delta_{ij}\partial_{t}h\\
&= \sigma_{k}^{ij}\nabla_{ij}F - {\eta(t)}\sigma_{k}^{ij}\nabla_{ij}h + \sigma_{k}^{ij}\delta_{ij}F - {\eta(t)}\sigma_{k}^{ij}\delta_{ij}h\\
&= \sigma_{k}^{ij}\nabla_{ij}F + \sigma_{k}^{ij}\delta_{ij}F - k\sigma_{k}{\eta(t)}.
\end{split}
\end{equation*}
The above equality and  $\eqref{req}$ gives that
\begin{equation*}
\begin{split}
\partial_{t} F &= (2-p)f(x) h^{1-p}\rho^{q-n}{\sigma_k} \partial_{t}h + (q-n)f(x) h^{2-p}{\rho^{q-n-1}\sigma_k}\partial_{t}\rho + f(x) h^{2-p}\rho^{q-n}\partial_{t}\sigma_{k}\\
& = (q-n+2-p)\frac{F^{2}}{h} - (q-n+k+2-p)\eta(t)F + NF\sigma_{k}^{ij}\delta_{ij}{+N\sigma_k^{ij}\nabla_{ij}F}.
\end{split}
\end{equation*}

Next, we consider the evolution equation of $\frac{F}{h} = f(x)h^{1-p}\rho ^{q-n} \sigma_{k}$.
\begin{equation}\label{Feq}
\begin{split}
&\partial_{t} \left( \frac{F}{h}\right) - N\sigma_{k}^{ij}\nabla_{ij}\left( \frac{F}{h}\right) \\
& = \frac{1}{h}(\partial_{t} F - N\sigma_{k}^{ij}\nabla_{ij}F) - \frac{F}{h^{2}}(\partial_{t} h - N\sigma_{k}^{ij}\nabla_{ij}h) +
2\frac{N}{h}\sigma_{k}^{ij}\nabla_{i}\left( \frac{F}{h}\right)\nabla_{j}h\\
& = (q-n+k+1-p)\left( \frac{F}{h}\right)^{2} - (q-n+k+1-p)\eta(t)\left( \frac{F}{h}\right) \\
& + 2\frac{N}{h}\sigma_{k}^{ij}
\nabla_{i}\left( \frac{F}{h}\right)\nabla_{j}h.
\end{split}
\end{equation}
Since $\eta(t)$ is uniformly bounded and $q-n+k+1-p < 0$, we have $\frac{F}{h}$ is bounded from below and above. The uniform bounds on $\sigma_{k}$ follow.
\end{proof}

Now we can derive the  upper bounds of principal curvatures $\kappa_{i}(x,t)$ of $M_t$
for $i=1,\cdots, n-1$.

\begin{lemma}\label{C2-estimate-2}
   Under the assumptions of Theorem \ref{thm1},
we have $$ \kappa_{i} \leq C, \quad \forall (x,t) \in \mathbb{S}^{n-1} \times [0, T),$$
where $C$ is a positive constant independent of $t$.
\end{lemma}
\begin{proof}
By rotation, we assume that the maximal eigenvalue of $\dfrac{b^{ij}}{h}$ at $t$  is attained at point $x_{t}$ in the direction
 of the unit vector $e_{1} \in T_{x_{t}} \uS$. We also choose orthonormal vector field such that $b^{ij}$ is diagonal.

 By the evolution equation of $b^{ij}$ in Lemma \ref{evolutions}, we have
\begin{equation*}
\begin{split}
& \partial_{t}\left(\frac{b^{11}}{h}\right) -  N\sigma_{k}^{ij}\nabla_{ij}\left(\frac{b^{11}}{h}\right)  \\
& = \frac{2}{h}N\eta(t)\sigma_{k}^{ij} \nabla_{i}\frac{b^{11}}{h}
\nabla_{j}h + \frac{N}{h^{2}} b^{11}\sigma_{k}^{ij}\nabla_{ij}h
- (k+1) \frac{N}{h}\sigma_{k}(b^{11})^{2} +  \frac{N}{h}\sigma_{k}^{ij}\delta_{ij}b^{11} \\
& - \frac{N}{h}(b^{11})^{2}(\sigma_{k}^{ij,mn} + 2\sigma_{k}^{im}b^{nj})\nabla_{1}b_{ij}\nabla_{1}b_{mn}\\
& -\frac{1}{h}(b^{11})^{2}( \nabla_{11}N\sigma_{k}  + 2\nabla_{1}\sigma_{k}\nabla_{1}N
) - \frac{b^{11}}{h^{2}}N\sigma_{k}+ \frac{2b^{11}}{h}{\eta(t)}\\
&= \frac{2}{h}N\sigma_{k}^{ij} \nabla_{i}\frac{b^{11}}{h}
\nabla_{j}h - (k+1) \frac{N}{h}\sigma_{k}(b^{11})^{2}  - \frac{N}{h}(b^{11})^{2}(\sigma_{k}^{ij,mn} + 2\sigma_{k}^{im}b^{nj})\nabla_{1}b_{ij}\nabla_{1}b_{mn}\\
& -\frac{1}{h}(b^{11})^{2}(\nabla_{11}N\sigma_{k}  + 2\nabla_{1}\sigma_{k}\nabla_{1}N)
+ (k-1) \frac{b^{11}}{h^{2}}N\sigma_{k}+ \frac{2b^{11}}{h}\eta(t).
\end{split}
\end{equation*}
According to inverse concavity of $(\sigma_{k})^{\frac{1}{k}}$, we obtain by  \cite{Urbas1991}
\beqs
(\sigma_{k}^{ij,mn} + 2\sigma_{k}^{im}b^{nj})\nabla_{1}b_{ij}\nabla_{1}b_{mn} \geq \frac{k+1}{k}\frac{(\nabla_{1}\sigma_{k})^{2}}{\sigma_{k}}.
\eeqs
On the other hand, by Schwartz inequality, the following inequality holds
\beqs
2|\nabla_{1}\sigma_{k}\nabla_{1}N| \leq \frac{k+1}{k}\frac{N(\nabla_{1}\sigma_{k})^{2}}{\sigma_{k}} + \frac{k}{k+1}\frac{\sigma_{k}(\nabla_{1}N)^{2}}{N}.
\eeqs
Hence,  we have at $(x_{t}, t)$
\begin{equation*}
\partial_{t}\frac{b^{11}}{h}\leq  -\frac{(b^{11})^{2}}{h}\sigma_{k}\left[\nabla_{11}N -
\frac{k}{k+1}\frac{(\nabla_{1}N)^{2}}{N} + (k+1)N  + (1-k)\frac{Nb_{11}}{h}\right]+ \frac{2b^{11}}{h}\eta(t).
\end{equation*}
Let $s$ be the arc-length of the great circle  passing through $x_{t}$ with the unit tangent vector $e_{1}$.
Notice that
\begin{equation*}
\nabla_{11}N -\frac{k}{k+1}\frac{(\nabla_{1}N)^{2}}{N} + (k+1)N = (k+1)N^{\frac{k}{k+1}}\left(N^{\frac{1}{k+1}} + (N^{\frac{1}{k+1}})_{ss}\right).
\end{equation*}
It is easy to compute
\begin{equation*}
\begin{split}
N_{s} & = f_{s} h^{2-p}\rho^{q-n} + (2-p)f h^{1-p}\rho^{q-n}h_{s} + (q-n)fh^{2-p}\rho^{q-n-2}(hh_{s} + h_{ss}h_{s})\\
N_{ss} & = f_{ss}h^{2-p}\rho^{q-n} + 2(2-p) h^{1-p}\rho^{q-n}f_{s} h_{s} + 2(q-n)h^{2-p}\rho^{q-n-2}f_{s} h_{s}b_{ss} \\
&+ 2(q-n)(2-p)fh^{1-p}\rho^{q-n-2}h^{2}_{s}b_{ss} + (2-p)(1-p)fh^{-p}\rho^{q-n}h_{s}^{2}\\
& + (q-n)(q-n-2)fh^{2-p}\rho^{q-n-4}b_{ss}^{2} h_{s}^{2} + (2-p)fh^{1-p}\rho^{q-n}{h_{ss}}\\
& + (q-n)f h^{2-p}\rho^{q-n-2}(b_{ss}^{2} - hb_{ss} + h_{s}^{2} + h_{ss m}h_{m}).
\end{split}
\end{equation*}

We have by direct computations
\begin{equation*}
\begin{split}
& 1+ N^{-\frac{1}{k+1}}\left(N^{\frac{1}{k+1}}\right)_{ss}  \\
 &= 1 + \frac{1}{k+1}N^{-1}N_{ss} - \frac{k}{(k+1)^{2}}N^{-2}N_{s}^{2}\\
&= 1 + \frac{1}{k+1}f_{ss}f^{-1}+ \frac{1}{k+1}2(2-p) h^{-1}f^{-1}f_{s} h_{s} +\frac{1}{k+1} 2(q-n)f^{-1}\rho^{-2}f_{s} h_{s}b_{ss} \\
&+ \frac{1}{k+1}2(q-n)(2-p)h^{-1}\rho^{-2}h^{2}_{s}b_{ss} + \frac{1}{k+1}(2-p)(1-p)h^{-2}h_{s}^{2}\\
& + \frac{1}{k+1}(q-n)(q-n-2)\rho^{-4}b_{ss}^{2} h_{s}^{2} + \frac{1}{k+1}(2-p)h^{-1}{h_{ss}}\\
& + \frac{1}{k+1}(q-n) \rho^{-2}(b_{ss}^{2} - hb_{ss} + h_{s}^{2} + h_{ss m}h_{m})\\
&- \frac{k}{(k+1)^{2}}f_{s}^{2}f^{-2} - \frac{k}{(k+1)^{2}} (2-p)^{2}h^{-2}h_{s}^{2} - \frac{k}{(k+1)^{2}} (q-n)^{2}\rho^{-4}b_{ss}^{2} h_{s}^{2}\\
& - \frac{k}{(k+1)^{2}} 2(2-p)f^{-1}h^{-1}f_{s} h_{s} - \frac{k}{(k+1)^{2}} 2(q-n)f^{-1}\rho^{-2}f_{s} h_{s}b_{ss} \\
& - \frac{k}{(k+1)^{2}} 2(2-p)(q-n)h^{-1}\rho^{-2}h_{s}^{-2}b_{ss}.\\
 &= 1 + \frac{1}{k+1}f_{ss}f^{-1}- \frac{k}{(k+1)^{2}}f_{s}^{2}f^{-2} + \frac{2-p}{k+1}h^{-1}{h_{ss}}\\
 & + \frac{2(2-p)}{(k+1)^{2}} h^{-1}f^{-1}f_{s} h_{s} + \frac{2(q-n)}{(k+1)^{2}}f^{-1}\rho^{-2}f_{s} h_{s}b_{ss}\\
 &+ \frac{2}{(k+1)^{2}} (2-p)(q-n)h^{-1}\rho^{-2}h^{2}_{s}b_{ss} + \frac{1}{(k+1)^{2}}(p-2)(p+k-1)h^{-2}h_{s}^{2}\\
 & + \frac{q-n}{(k+1)^{2}}(q-n-2k-2)\rho^{-4}b_{ss}^{2} h_{s}^{2} + \frac{1}{k+1}(q-n) \rho^{-2}(b_{ss}^{2} - hb_{ss} + h_{s}^{2} + h_{ss m}h_{m}).
\end{split}
\end{equation*}
At $(x_{t}, t)$, we have
\begin{equation*}
0 = \nabla_{m}\frac{b^{11}}{h} = \frac{\nabla_{m}b^{11}}{h} - \frac{b^{11}}{h^{2}}h_{m}.
\end{equation*}
Then
\begin{equation*}
\begin{split}
 h_{ss m}h_{m} & = (b_{ss } - h)_{m}h_{m} = b_{ssm }h_{m} - h_{m}^{2}\\
 &= -\frac{b_{ss }}{h}h_{m}^{2} - h_{m}^{2}.
\end{split}
\end{equation*}
From this we get
\begin{equation*}
\begin{split}
& b_{ss}^{2} - hb_{ss} + h_{s}^{2} + h_{ss m}h_{m} \\
&= b_{ss}^{2} - hb_{ss} + h_{s}^{2} -\frac{b_{ss }}{h}h_{m}^{2} - h_{m}^{2}\\
& = b_{ss}^{2} + h_{s}^{2} - \frac{b_{ss }}{h}\rho^{2} - \rho^{2} + h^{2}.
\end{split}
\end{equation*}
Now, we have
\begin{equation*}
\begin{split}
& 1+ N^{-\frac{1}{k+1}}\left(N^{\frac{1}{k+1}}\right)_{ss}  \\
 &= 1 + \frac{1}{k+1}f_{ss}f^{-1}- \frac{k}{(k+1)^{2}}f_{s}^{2}f^{-2} + \frac{2-p}{k+1}h^{-1}h_{ss}\\
 & + \frac{2(2-p)}{(k+1)^{2}} h^{-1}f^{-1}f_{s} h_{s} + \frac{2(q-n)}{(k+1)^{2}}f^{-1}\rho^{-2}f_{s} h_{s}b_{ss}\\
 &+ \frac{2}{(k+1)^{2}} (2-p)(q-n)h^{-1}\rho^{-2}h^{2}_{s}b_{ss} + \frac{1}{(k+1)^{2}}(p-2)(p+k-1)h^{-2}h_{s}^{2}\\
 & + \frac{q-n}{(k+1)^{2}}(q-n-2k-2)\rho^{-4}b_{ss}^{2} h_{s}^{2} \\
 & - \frac{1}{k+1}(q-n)+ \frac{1}{k+1}(q-n) \rho^{-2}(b_{ss}^{2} - \frac{b_{ss }}{h}\rho^{2} + h_{s}^{2} + h^{2}).
\end{split}
\end{equation*}
Since $1 > \rho^{-2} h_{s}^{2}$ and $q > n+k+1$, we have
\begin{equation*}
\begin{split}
& \frac{2(q-n)}{(k+1)^{2}}f^{-1}\rho^{-2}f_{s} h_{s}b_{ss} + \frac{q-n}{(k+1)^{2}}(q-n-2k-2)\rho^{-4}b_{ss}^{2} h_{s}^{2} + \frac{1}{k+1}(q-n) \rho^{-2}b_{ss}^{2}\\
& > \frac{2(q-n)}{(k+1)^{2}}f^{-1}\rho^{-2}f_{s} h_{s}b_{ss} + \frac{q-n}{(k+1)^{2}}(q-n-k-1)\rho^{-4}b_{ss}^{2} h_{s}^{2}\\
& = \frac{q-n}{(k+1)^{2}}\left((q-n-k-1)^{\frac{1}{2}}\rho^{-2} h_{s}b_{ss} + f_{s}f^{-1}\frac{1}{(q-n-k-1)^{\frac{1}{2}}} \right)^{2}\\
& - \frac{q-n}{(k+1)^{2}}\frac{1}{q-n-k-1}f_{s}^{2}f^{-2}.
\end{split}
\end{equation*}
We also have
\begin{equation*}
\begin{split}
&  \frac{2(2-p)}{(k+1)^{2}} h^{-1}f^{-1}f_{s} h_{s}  +  \frac{1}{(k+1)^{2}}(p-2)(p+k-1)h^{-2}h_{s}^{2}\\
 & = \frac{(p-2)}{(k+1)^{2}}\left( (p+k-1)^{\frac{1}{2}}h^{-1}h_{s} - \frac{1}{(p+k-1)^{\frac{1}{2}}} f_{s}f^{-1} \right)^{2}\\
 & - \frac{(p-2)}{(k+1)^{2}}\frac{1}{p+k-1}f_{s}^{2}f^{-2}.
\end{split}
\end{equation*}
 Noticing that $p \geq 2$, we have
\begin{equation*}
\begin{split}
& 1+ N^{-\frac{1}{k+1}}\left(N^{\frac{1}{k+1}}\right)_{ss}  \\
 &\geq \frac{p+n-q+k-1}{k+1} + \frac{2-p-q+n}{k+1}h^{-1}b_{ss} + \frac{2}{(k+1)^{2}}(2-p)(q-n)\rho^{-2}h^{-1}b_{ss} h_{s}^{2}\\
 &+ \frac{1}{k+1}f_{ss}f^{-1} - \left(\frac{k}{(k+1)^{2}}  + \frac{1}{(k+1)^{2}} \frac{q-n}{q-n-k-1} + \frac{1}{(k+1)^{2}}\frac{p-2}{p+k-1}\right)f_{s}^{2}f^{-2}\\
 & = \frac{2-p-q+n}{k+1}h^{-1}b_{ss} + \frac{2}{(k+1)^{2}}(2-p)(q-n)\rho^{-2}h^{-1}b_{ss} h_{s}^{2}\\
 & + \frac{1}{k+1}f^{-1}\left( f_{ss} - \frac{1}{k+1} \left(k  +  \frac{q-n}{q-n-k-1}+ \frac{p-2}{p+k-1}\right)f_{s}^{2}f^{-1} + (p+k-1-q+n)f\right)
\end{split}
\end{equation*}
Since $f$ satisfies ${\rm (A)}$, we have
\begin{equation*}
\partial_{t} \frac{b^{11}}{h}  \leq  -\left(\frac{b^{11}}{h}\right)^{2}N\sigma_{k}( c_{f}h - c_{0}b_{11}) + \frac{2b^{11}}{h}{\eta(t)}.
\end{equation*}
Here $c_{f}$ is a positive constant depending on $f$ and $c_{0}$ is a positive constant depending on the uniform bounds of $h$ and $|\nabla h|$.
By the uniform bounds on $h$,  $f$, $\eta$ and $\sigma_{k}$, we conclude
\begin{equation*}
\partial_{t}\frac{b^{11}}{h} \leq  -c_{1}\left(\frac{b^{11}}{h}\right)^{2} + c_{2}\frac{b^{11}}{h}.
\end{equation*}
Here $c_{1}$  and $c_{2}$ are  positive constants independent of  $t$.
The maximum principle then gives the upper bound of  $b^{11} $ and the result follows.
\end{proof}

\begin{proof}[proof of Theorem 1]
Combining Lemma \ref{sigma-lower-estimate} and Lemma \ref{C2-estimate-2}, we see that
the principal curvatures of $M_{t}$ has uniform positive upper and lower bounds.
This together with Lemma \ref{C0-estimate} and Corollary \ref{cor3.2} implies that the
evolution equation \eqref{seq} is uniformly parabolic on any finite time
interval. Thus, the result of \cite{KS1980} and the standard
parabolic theory show that the smooth solution of \eqref{seq} exists for all
time.
And by these estimates again, a subsequence of $M_t$ converges in $C^\infty$ to
a positive, smooth, strictly convex hypersurface $M_\infty$ in $\R^n$.
\end{proof}

In general, the problem $\eqref{seq}$ does not
have any variational structure, we can not expect the convergence of the flow to a solution for all initial hypersurfaces. In the next section, we will
choose an special to obtain the existence of smooth solutions for equation $\eqref{seq}$ with $c=1$.

\section{Existence of solution}

Since $q-n+k+1-p < 0$, we can choose initial hypersurface satisfying
$\left( \frac{F}{h} -1\right)_{M_{0}} > 0$. We have proved that flow $\eqref{feq}$ exists for all time when  $\eta =1$.
We will use flow $\eqref{seq}$ with $\eta =1$ to obtain the convergence.

 By $\eqref{Feq}$, we have
\begin{equation*}
\begin{split}
&\partial_{t} \left( \frac{F}{h} -1\right) - N\sigma_{k}^{ij}\nabla_{ij}\left( \frac{F}{h} -1\right) \\
& = (q-n+k+1-p)\left( \frac{F}{h}-1\right)\frac{F}{h} + 2\frac{N}{h}\sigma_{k}^{ij}
\nabla_{i}\left( \frac{F}{h}\right)\nabla_{j}\left( \frac{F}{h}-1\right).
\end{split}
\end{equation*}
From the assumption about the initial data, the positivity of $\frac{F}{h}-1$ is preserved along the flow and
\begin{equation*}
\partial_{t}h  = F - h > 0
\end{equation*}
holds for all time. Since

Since $h$  is positive and bounded  from above and below, we have
\begin{equation*}
C \geq h(x,t) - h (x,0) = \int_{0}^{+\infty} (F - h) dt
\end{equation*}
 This implies that there exists a subsequence of times $t_j\to\infty$ such that
\begin{equation*}
 F({t_j}) - h({t_j}) \to 0 \text{ as } t_j\to\infty.
\end{equation*}

 By
passing to the limit, we obtain
\begin{equation*}
 \tilde{h}^{1-p}\widetilde{\rho}^{q-n} \widetilde{\sigma_{k}}(x) f(x)=  1\text{ on } \uS,
\end{equation*}
where $\widetilde{\sigma_{k}}$, $\tilde{h}$ and $\widetilde{\rho}$ are  the
$k$-th  elementary symmetric function for principal curvature radii, the support function  and radial function of $M_\infty$. The proof of
Corollary \ref{coreq} is complected.







\begin{thebibliography}{99}



\bibitem{BF2019}
\newblock K. J. Boroczky and F. Fodor,
\newblock {{$L_p$} dual {M}inkowski prolblem for $p > 1$ and $q >
0$},
\newblock \emph{ J. Differential Equations},  \textbf{266}  (2019), 7980--8033.

\bibitem{BIS2019}
 \newblock P. Bryan, M.~N. Ivaki  and J. Scheuer,
 \newblock { A unified flow approach to
  smooth, even {$L_p$}-{M}inkowski problems},
  \newblock \emph{ Anal. PDE}, \textbf{12} (2019),
  259--280.

\bibitem{BIS2021}
\newblock P. Bryan, M.~N. Ivaki  and J. Scheuer,
 \newblock { Orlicz-Minkowski flows},
 \newblock \emph{ Calc. Var. Partial Diff. Equ.}, \textbf{60}, 41(2021).

\bibitem{Chen&Huang&Zhao2018}
\newblock C. Q. Chen, Y. Huang and  Y. M. Zhao,
 \newblock {Smooth solutions to the $L_{p}$ dual Minkowski problems},
 \newblock \emph{ Math. Ann.}, \textbf{373}(2019), 953--976.
















 \bibitem{GuanMa2003}
 \newblock P. F. Guan and X. N. Ma,
\newblock  {The Christoffel-Minkowski problem. I. Convexity of solutions of Hessian equation},
\newblock \emph{Invent. Math.}, \textbf{151} (2003), 553--571.

 \bibitem{GuanXia2018}
 \newblock P. F. Guan and C. Xia,
 \newblock  {{$L^p$} Christoffel-Minkowski problem: the case $1 < p < k+1$},
 \newblock \emph{ Calc. Var. Partial Diff. Equ.}, \textbf{57}, 69
 (2018).





 \bibitem{HMS2004}
 \newblock C. Hu, X. N. Ma and C.  Shen,
 \newblock  {On Christoffel-Minkowski problem of Firey's $p$-sum},
 \newblock \emph{Calc. Var. Partial Diff. Equ.},
 \textbf{21}(2004), 137--155.


 \bibitem{HuangActa2016}
 \newblock Y. Huang, E. Lutwak, D. Yang and G. Zhang,
  \newblock  { Geometric measures in the dual Brunn-Minkowski theory and their associated Minkowski problems},
   \newblock \emph{Acta Math.}, \textbf{216} (2016), 325--388.

 \bibitem{HuangZhao2018}
 \newblock Y. Huang and Y. M. Zhao,
 \newblock  {On the {$L_p$} dual Minkowski problem},
 \newblock \emph{ Adv. Math.} \textbf{332} (2018), 57--84.


\bibitem{Ivaki2019}
\newblock M. N. Ivaki,
\newblock  { Deforming a hyper surface by principal radii of curvature and support function},
 \newblock \emph{Calc. Var. Partial Diff. Equ.}, \textbf{58}(1)(2019), 2133--2165.



 \bibitem{KS1980}
\newblock N. V. Krylov and  M. V. Safonov,
\newblock  {A property of the solutions of
  parabolic equations with measurable coefficients},
  \newblock \emph{Izv. Akad. Nauk SSSR Ser.
  Mat.}, \textbf{239} (44) (1980), 161--175.



\bibitem{LiuLuTrans}
\newblock  Y. N. Liu  and J. Lu,
\newblock   {A flow method for the dual {O}rlicz-{M}inkowski problem},
\newblock \emph{ Trans. Amer. Math. Soc.}, \textbf{373} (2020), 5833--5853.


\bibitem{LutwakYangZhang2018}
\newblock E. Lutwak, D. Yang and G. Y. Zhang,
\newblock  { $L_{p}$ dual curvature measures},
\newblock \emph{ Adv. Math.}, \textbf{329} (2018), 85-132.

 \bibitem{Schneider}
 \newblock R. Schneider,
  \newblock  { Convex bodies, the Brunn-Minkowski theory},
   \newblock vol. 151 of Encyclopedia of Mathematics and its Applications, Cambridge University Press, Cambridge, expanded, 2014.


\bibitem{ShengYi2020}
\newblock W. M. Sheng and C. H. Yi,
 \newblock { A class of anisotropic expanding curvature flow},
 \newblock \emph{ Disc.  Conti. Dynam. Systems-A}, \textbf{40}(2020), 2017-2035.


\bibitem{Urbas1991}
\newblock J. Urbas,
\newblock  {An expansion of convex hypersurfaces},
\newblock \emph{ J. Diff. Geom.}, \textbf{33} (1991), 91-125.


\end{thebibliography}
\end{document}